\documentclass[a4paper,final]{amsart}
\usepackage{graphicx} 
\usepackage{amsfonts}
\usepackage{amsmath}
\usepackage{hyperref}
\usepackage{tikz}
\usepackage{pgfplots}
\pgfplotsset{compat=1.15}
\usepackage{mathrsfs}
\usetikzlibrary{arrows}

\usepackage{amssymb}
\usepackage{enumerate}
\usepackage{amsmath}
\usepackage{amssymb}
\usepackage{framed,xcolor}
\usepackage{lipsum,hyperref}

\colorlet{shadecolor}{yellow!20}

\numberwithin{equation}{section}
\newtheorem{theorem}{Theorem}[section]
\newtheorem{lemma}[theorem]{Lemma}
\newtheorem{claim}[theorem]{Claim}

\newtheorem{cor}[theorem]{Corollary}

\newcommand{\thistheoremname}{}
\newtheorem*{genericthm*}{\thistheoremname}
\newenvironment{namedthm*}[1]
  {\renewcommand{\thistheoremname}{#1}%
   \begin{genericthm*}}
  {\end{genericthm*}}

\theoremstyle{definition}       {         
\newtheorem{remark}[theorem]{Remark}

}

\newcommand{\R}{\mathbb{R}}

\def\beq{\begin{equation}}
\def\eeq{\end{equation}}
\newcommand{\semmi}[1]{{}}

\newcommand{\bp}{\begin{proof}}
\newcommand{\ep}{\end{proof}}

\renewcommand{\phi}{\varphi}

\newcommand{\mj}[1]{}

\author{Márk Kökényesi}
\title{Simultaneous Rotation of Infinitely Many Parallel Line Segments}

\begin{document}

\address
{Institute of Mathematics, ELTE E\"otv\"os Lor\'and University, 
P\'az\-m\'any P\'e\-ter s\'et\'any 1/c, H-1117 Budapest, Hungary}

\email{mark.p.kokenyesi@gmail.com}



\maketitle

\begin{abstract}
    In 1971, Davies proved that finitely many parallel line segments can be simultaneously fully rotated in an arbitrarily small area.
    In this paper we show that an even stronger statement holds: The unit square can be fully rotated in such a way that each initially vertical line segment sweeps a set of small area.
    
    A set in $\mathbb{R}^n$ is said to have the strong Kakeya property if for any two of its positions, the set can be continuously moved between these two positions in an arbitrarily small volume.
    We use the above result to show that a wide family of sets in $\mathbb{R}^3$, for instance the lateral surface of a cylinder, have the strong Kakeya property.
\end{abstract}

\section{Introduction}

In 1917, Kakeya posed a question now known as the Kakeya needle problem: Among sets, in which a line segment of unit length can be fully rotated, which one has minimal area? Besicovitch \cite{besi} showed that there is no such minimum, the area of such a set can be arbitrarily small. In 1971, Davies \cite{Da} demonstrated that more is true: finitely many parallel line segments can be simultaneously rotated, in an arbitrarily small area. Our main result is even stronger: 

\begin{theorem}\label{t:mainresult}
    For every $\varepsilon>0$ there exists a continuous motion of the unit square during which every initially vertical line segment sweeps at most $\varepsilon$ area, while the square does a full rotation.
\end{theorem}

A set in $\mathbb{R}^n$ is said to have the Kakeya property if it has two positions, between which the set can be continuously moved in an arbitrarily small volume \cite{kapo}. It is said to have the strong Kakeya property if this applies to any two of its positions.

It is easy to show that the result of Davies \cite{Da} implies that a finite union of parallel line segments possess the strong Kakeya property. In the same paper he showed that the segments must be parallel, two non-parallel segments do not have the strong Kakeya property.

Much effort has been put into characterizing sets in $\mathbb{R}^2$ with the Kakeya property and the strong Kakeya property. M. Cs{\"o}rnyei, K. H{\'e}ra, and M. Laczkovich \cite{kapo} showed that closed connected set with the Kakeya property must be a subset of a line or circle. K. H{\'e}ra and M. Laczkovich \cite{hera2016} proved that short enough circular arcs do in fact possess the strong Kakeya property. Currently the strongest version of this statement is due to A. Chang and M. Cs{\"o}rnyei \cite{korka}. They generalized the result of Davies in another direction by showing that a countable union of parallel line segments which is bounded and has finite total length can be rotated in an arbitrarily small area. It is still unknown whether all non-complete circular arcs possess the strong Kakeya property.

Let $V$ be the set of planes, which have a $45^\circ$ angle with the $\{y=0\}$ plane. In Section~\ref{s:2} we will prove the following lemma.
\begin{lemma}\label{t:folemma}
    There exists a closed set $A$ in $\mathbb{R}^3$, which is the union of planes in $V$ at a bounded distance from the origin, contains a translate of every plane in $V$, and intersects every plane of the form $\{y=t\}$ in a set of measure zero.
\end{lemma}

The proof relies on a duality argument and utilizes the construction of Talagrand \cite{talagrand}, for an English description see Appendix A. in \cite{HANSEN}.

Lemma~\ref{t:folemma} is closely related to the question of rotating a square. The connection can be seen if we think of the plane as the $\{y=0\}$ plane in $\mathbb{R}^3$ and raise every vertical segment of the unit square by its distance from the origin. During the motion of the unit square, the motion of each initially vertical segment corresponds to the motion of the raised rectangle at a given height.

Thus Lemma~\ref{t:folemma} can be thought of as giving a discontinuous motion of the unit square, during which it achieves every direction. In Section~\ref{s:rot} we turn this motion into a continuous one and prove Theorem~\ref{t:mainresult}.

Section~\ref{s:sk} will discuss how Theorem~\ref{t:mainresult} can be applied to show that certain sets in $\mathbb{R}^3$ have the strong Kakeya property. 
To demonstrate the connecting idea, consider the lateral surface of a cylinder with an axis parallel to the $z$-axis. If we look at this set from the direction of the $y$-axis, then we see a subset of a square. We can apply a motion given by Theorem~\ref{t:mainresult} that rotates this square by $90^\circ$. This will define a motion on the surface as well. During the motion of the surface the area swept in the $\{ y=h \}$ plane will be at most the sum of the areas swept by the segments in this plane, which is less than $\varepsilon$ for each segment. Therefore, this motion moved the surface of the cylinder with an axis parallel to the $z$-axis to one with an axis parallel to the $x$-axis in an arbitrarily small volume. This idea applies to a much wider family of sets, not just the lateral surface of a cylinder.

We say that a compact set $K\subset \mathbb{R}^3$ is \textit{cylinderlike} if there exists $n\in \mathbb{N}$ such that for almost all $t$ the set $\{y=t\}\cap K$ can be covered by $n$ vertical lines. We say that a compact set $K$ is \emph{cylinderlike from direction $d$} if there is an orthogonal coordinate system, in which the $y$-axis is parallel to $d$, and $K$ is cylinderlike.

\begin{theorem}\label{t:fogantyu}
     If $K$ is cylinderlike from two non-parallel directions $d_1, d_2$, then $K$ has the strong Kakeya property.
\end{theorem}

It follows that the union of a finite number of parallel lateral cylinder surfaces possesses the strong Kakeya property. The above theorem also implies that a compact subset of any two planes has the strong Kakeya property.

More general examples of such sets are given at the end of Section~\ref{s:sk}.

\section{Planes of certain angle, proof of Lemma~\ref{t:folemma}}\label{s:2}
This section will exclusively deal with the proof of Lemma~\ref{t:folemma}.

Let $a\cdot b$ denote the standard dot product of $a$ and $b$.
For a vector $v$ let \[p_v(X)~=~\{x \cdot v~|~x~\in~X\}.\]

It is clear that it is sufficient to construct a set fulfilling the conditions, with the exception that it only contains planes in $V$, whose directions form an interval.

We will use a duality argument and encode planes with points of $\mathbb{R}^3$. Similar encodings are often used in planar cases, see for examples \cite{tomi},\cite{falconer}. To the point $(a,b,c)$ assign the plane, which contains $(0,0,a)$, has slope $b$ in the $y=0$ plane, slope $c$ in the $x=0$ plane. It is easy to see that a point $(x,y,z)$ is on the plane corresponding to $(a,b,c)$ if and only if $a+bx+cy=z$. Through straightforward calculation a plane is in $V$ when its triple has $c^2 - b^2 = 1$.

\subsection{Requirements for our codeset and some basic observations} Let $K$ be a set of triples, and $A$ the union of the corresponding planes. If $K$ is compact, then $A$ is closed and every plane is at a bounded distance from the origin. The point $(x,y,z)$ is in $A$ if $K$ contains a triple $(a,b,c)$, which satisfies $a+bx+cy=z$, meaning $K$ has a point in this plane. The set $(x,y, - ) \cap A$ is $p_{(1,x,y)}(K)$, hence their measure is equal. Hence if $K$ is such that for all vectors $v$ of the form $(1,x,y)$ we have $\lambda(p_v(K)) = 0$, then $A$ intersects every vertical line in a set of measure 0. Therefore, it intersects every plane $\{y=t\}$ in a set of measure 0. For the directions of the planes in $A$ to form an interval, we need $p_{(0,1,0)}(K)$ to be an interval.

Let the 3 coordinates of the space containing $K$ be $a,b,c$. By the above argument in order to prove Lemma~\ref{t:folemma}, it is enough to prove the following lemma.

Let $H= \{(a,b,c) | c^2 - b^2 = 1, c > 0\}$.

\begin{lemma}\label{t:letK}
    There exists a compact set $K \subset H$ such that $ p_{(0,1,0)}(K)$ is an interval and $ \lambda(p_{(1,x,y)})(K) = 0$ for all $x,y$.
\end{lemma}

The remainder of this section deals with the proof of Lemma \ref{t:letK}.

\subsection{Preliminaries for constructing $K$}

Let $f: \mathbb{R}^2 \rightarrow H$, $f(a,b)= (a,b, \sqrt{1+b^2})$.

We will construct $K' \subset \mathbb{R}^2$ such that $K= f(K')$ has the desired properties.

For any $x,y,s$ we have $s \in p_{(1,x,y)}(K)$ if and only if $K$ has a point $r$, for which
\[r\cdot(1,x,y) = s.\]
Such points $r$ on $H$ form a curve: 
\[c_{x,y,s}= \{(s-x\sinh(t)-y\cosh(t),\sinh(t),\cosh(t)):t \in \mathbb{R}\}.\]
Define
\[ C_{x,y,s}= \{(s-x\sinh(t)-y\cosh(t),\sinh(t)):t \in \mathbb{R}\}.\]
Then
\[s \in p_{(1,x,y)}(K) \iff (K \cap c_{x,y,s} \neq \varnothing) \iff  (K' \cap C_{x,y,s} \neq \varnothing).\]

For each $x,y$ we define a function $\alpha_{x,y}$ on the plane: To calculate $\alpha_{x,y}(a,b)$, take the curve $C_{x,y,s}$ through $(a,b)$ and take the slope of the tangent of this curve at $(a,b)$ with respect to the $b$ axis. Observe that $\alpha_{x,y}$ always exists and is never $\pm \infty$. The function $\alpha_{(x,y)}(a,b)$ is continuous as a function of $x,y,a,b$. Thus, it is uniformly continuous on compact sets. Applying this to $[0,G]\times[0,G]\times[0,1]\times[0,1]$ we obtain

\begin{equation}\label{e:egyfoly}
\begin{aligned}
    \forall G \in \R \forall \varepsilon  \exists \delta \,& \forall r_1, r_2 \in [0, 1]^2, \, |x|, |y| \leq G, \, |r_1 - r_2| < \delta \\
    &\implies \left| \alpha_{x,y}(r_1) - \alpha_{x,y}(r_2) \right| < \varepsilon.
\end{aligned}
\end{equation}

\subsection{Construction of $K'$}

We will prove that the set constructed by Talagrand \cite{talagrand} (see also in \cite{HANSEN}) is a suitable set. For the sake of completeness, we repeat the construction.

By induction, we shall construct a decreasing sequence $(K_m')$ of finite unions of closed rectangles such that $K'=\cap_{m=1}^{\infty} K_m'$ has the desired properties. To begin with let \( \varepsilon_1 = 1 \) and \( K_1' = [0, 1]^2 \). 
We now suppose that \( m \in \mathbb{N}, 0 < \varepsilon_m \leq 1/m \), and that \( K_m' \) is a union of rectangles 
\[
R_n = [a_n, a_n + \frac{\varepsilon_m}{N}] \times \left[ \frac{n - 1}{N}, \frac{n}{N} \right], \quad 1 \leq n \leq N,
\]
where $N \in \mathbb{N}^+$, \( a_1, \dots, a_N \in \mathbb{R} \). We will now construct $K_{m+1}'$.

We fix \( 1 \leq n \leq N \), define \( P_{0,1} = R_n \), and choose \( k_m \in \mathbb{N} \) such that \( k_m \geq \frac{2m}{\varepsilon_m} \). Starting with \( P_{0,1} \), we construct parallelograms \( P_{i,j}, 1 \leq i \leq k_m, 1 \leq j \leq 2^i \). We take the midpoint of every side of $P_{i-1,j}$, let the midpoint of the left side be $A_{i,j}$, and the midpoint of the right side $B_{i,j}$. Take the midpoint of $A_{i,j}B_{i,j}$ and let the parallelograms $P_{i,2j-1}$, $P_{i,2j}$ be as indicated by Fig. \ref{fig:paragyartas}. Let $\alpha_i$ denote the slope corresponding to the highlighted angle of Fig. \ref{fig:paragyartas}.

\begin{figure}
    \centering
    \includegraphics[width=1\linewidth]{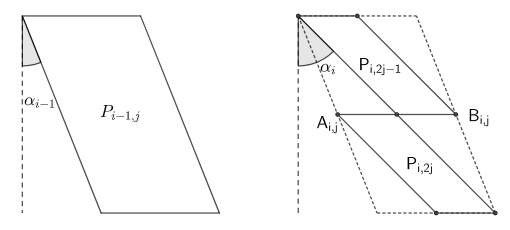}
    \caption{From $P_{i-1,j}$ to $P_{i,2j-1}$ and $P_{i,2j}$.}
    \label{fig:paragyartas}
\end{figure}

 It follows by induction from the construction that the two intervals \( p_{(0,1)}(P_{i,2j-1}) \) and \( p_{(0,1)}(P_{i,2j}) \) have length \( 2^{-i}/N \) and that \( |A_{i,j} - B_{i,j}| = 2^{-i+1}\varepsilon_m/N \). Therefore, in the last step also using that $k_m\ge \frac{2m}{\varepsilon_m}$, we obtain

\[
 \alpha_i = \frac{\varepsilon_m}{2} +  \alpha_{i-1},
\]
\begin{equation}\label{eq:szog}
    \implies \alpha_i = \frac{i\varepsilon_m}{2},
\end{equation}
\[
\implies \alpha_{k_m} \ge m.
\]

Next we shall replace each \( P_{k_m,j} \) by a union of rectangles \( Q_{k_m,j} \) (see Fig.~$\ref{fig:negyzetpak}$) such that \( p_{(0,1)}(Q_{k_m,j}) = p_{(0,1)}(P_{k_m,j}) \) and their union \( T(K_m') \) satisfies \[\lambda(p_{(1,0)}(T(K_m'))) \leq\frac{1}{m+1} .\] To that end we choose a suitable multiple \( N' \) of \( 2^{k_m} N \), an \( \varepsilon_{m+1} \in (0, \frac{1}{m+1})) \), which is sufficiently small, and replace each of the parallelograms \( P_{k_m,j}, 1 \leq j \leq 2^{k_m} \), by a subset \( Q_{k_m,j} \), which is a union of rectangles of the form \\
\([t,~t~+~\varepsilon_{m+1}/N']~\times~[u,~u~+~1/N']\) as indicated by Fig. 2. Using \eqref{e:egyfoly} we can choose $N'$ to be large enough, such that the oscillation of $\alpha_{x,y}$ is less than $\frac{\varepsilon_{m+1}}{2}$ in each rectangle, whenever $x,y < m$.

\begin{figure}
    \centering
    \includegraphics[width=1\linewidth]{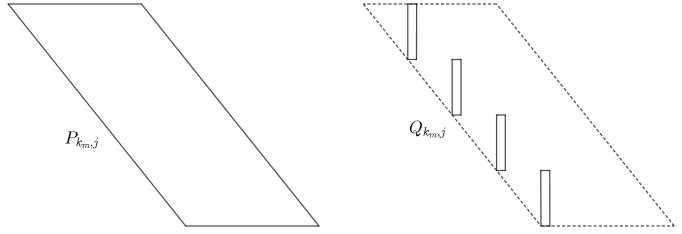}
    \caption{From $P_{k_m,j}$ to $Q_{k_m,j}$.}
    \label{fig:negyzetpak}
\end{figure}

If \( m \) is odd, let \( K_{m+1}' = T(K_m') \). If \( m \) is even, we take \( K_{m+1}' = S(T(S(K_m'))) \), where \( S \) denotes the reflection about the line \(\{b = 1/2\}\).

\subsection{Showing that $K$ is suitable}
Recall that $K = f(K')$.
It is clear that $K \subset H$ and $p_{(0,1,0)}(K)$ is an interval, so it remains be proved that $\lambda(p_{(1,x,y)}(K))=0$ for any $x,y$. We now fix $x,y$.
Let $F \subset \mathbb{R}^2$ be the set on which $\alpha_{x,y}$ is non-negative. It is easy to verify that $F$ is a half-plane with a defining line parallel to the $a$-axis, or the whole plane or possibly the empty set. We prove that
\begin{equation}\label{l:kff}
    \lambda(p_{(1,x,y)}(K \cap f(F))=0.
\end{equation}

The reasoning is similar if $\alpha_{x,y}$ is negative, in that case we need to use that for even $m$ we have \( K_{m+1}' = S(T(S(K_m'))) \). Thus \eqref{l:kff} implies that $K$ is a suitable set.

The following claim clearly implies \eqref{l:kff}, therefore it will complete the proof of Lemma \ref{t:letK} and thus the proof of Lemma \ref{t:folemma}.
\begin{claim}
    If m is large enough and even, then \[\lambda(p_{(1,x,y)}(f(K_m' \cap F)) \le \frac{3\sqrt{1+x^2+y^2}}{2(m-1)}.\]
\end{claim}

\begin{proof}
    Let $M$ be the maximum of $\alpha_{x,y}$ on $[0,1]^2$. We can choose $m$ to be larger than $x+2$ and $y+2$.
    The set $(K_{m-1}'\cap F)$ is a union of rectangles, this follows from the observation made on the possible shapes of $F$. Let $R$ be one of these rectangles. The oscillation of $\alpha_{x,y}$ on $R$ is less than $\frac{\varepsilon_{m-1}}{2}$, since $x,y < m -2$. Let $n$ be such that 
    \[\alpha_{x,y} \subset \left[\frac{n\varepsilon_{m-1}}2, \frac{(n+2)\varepsilon_{m-1}}2\right]\]
    on $R$. For large enough $m$ we have $n+2 < k_{m-1}$, since $k_{m-1} > \frac{2(m-1)}{\varepsilon_{m-1}}$ and $n < \frac{2M}{\varepsilon_{m-1}}$.

\begin{figure}
    \centering
    \includegraphics[width=0.75\linewidth]{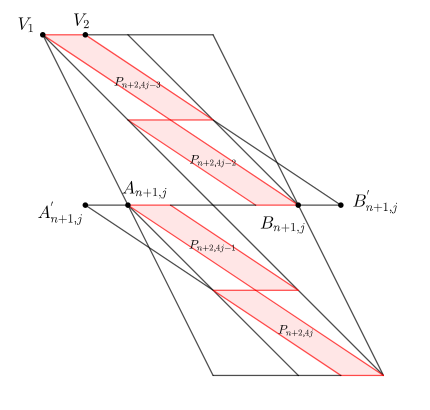}
    \caption{Two generations of parallelograms.}
    \label{fig:2gen}
\end{figure}

    Let $A_{n+1,j}' B_{n+1,j}'$ be the segment we obtain by scaling $A_{n+1,j}B_{n+1,j}$ by $\frac{3}{2}$ from its midpoint (Fig.~$\ref{fig:2gen}$). We claim that every curve $C_{x,y,s}$ intersecting $P_{n+2,4j-i}$, where $i = 0,1,2,3$ must intersect the segment $A_{n+1,j}^{'}B_{n+1,j}^{'}$. This can be verified for each parallelogram separately.
    We will only check this for $P_{n+2,4j-3}$, the others can be done similarly. Suppose $C_{x,y,s}$ intersects $P_{n+2,4j-3}$. We claim that if it intersects the line of $A_{n+1,j}^{'}B_{n+1,j}^{'}$ right of $B_{n+1,j}'$, then it must be right of the segment $B_{n+1,j}' V_2$ in the horizontal strip between $B_{n+1,j}'$ and $ V_2$. Indeed, this follows from the fact that by \eqref{eq:szog} the segment $B_{n+1,j}' V_2$ has a slope of $\frac{(n+2)\varepsilon_{m-1}}2$ with respect to the $b$-axis, since the tangent of the curve has a smaller slope on $R$. On the left side, a slightly stronger statement is true. If the curve intersects the line of $A_{n+1,j}^{'}B_{n+1,j}^{'}$ left of $A_{n+1,j}$, then it must be left of the segment $A_{n+1,j}V_1$ in the horizontal strip between $A_{n+1,j}$ and $V_1$ (this segment has a slope of $\frac{n\varepsilon_{m-1}}2$ with respect to the $b$-axis). 
    
    The segment $A_{n+1,j}^{'}B_{n+1,j}^{'}$ is of length $\frac{2^{-(n+1)}3\varepsilon_{m-1}}{N}$. Note that $f(A_{n+1,j}^{'}B_{n+1,j}^{'})$ is a segment of the same length.
    This means that
    \begin{equation*}
    \begin{aligned}
        \lambda(p_{(1,x,y)}(f( \cup_{i=0}^{3}& P_{n+2,4j-i}))) \le \lambda(p_{(1,x,y)}(f(A_{n+1,j}^{'}B_{n+1,j}^{'}))) \le   \\
        &\le \frac{2^{-(n+1)}3\varepsilon_{m-1}}{N}\sqrt{1+x^2+y^2},
    \end{aligned}
    \end{equation*}

\[\lambda(p_{(1,x,y)}(f(R\cap K_m'))) \le \sum_{j=1}^{2^{n}}\lambda(p_{(1,x,y)}(f( \cup_{i=0}^{3} P_{n+2,4j-i}))) \le \frac{3\varepsilon_{m-1}}{2N}\sqrt{1+x^2+y^2},\]

\begin{equation*}
\begin{aligned}
    &\lambda(p_{(1,x,y)}(f(K_{m}'\cap F))) \le \sum_{i=1}^N \lambda(p_{(1,x,y)}(f(R_i))) \le \\
    &\le \frac{3\varepsilon_{m-1}}{2}\sqrt{1+x^2+y^2} \le \frac{3\sqrt{1+x^2+y^2}}{2(m-1)}.
\end{aligned}
\end{equation*}

\end{proof}

\section{Rotating a square, the proof of Theorem~\ref{t:mainresult}}\label{s:rot}
\subsection{Neighbourhoods of $A$}
Call a rectangle in $\mathbb{R}^3$ $\textit{interesting}$ if its sides are of length $1$ and $\sqrt{2}$, one of the shorter sides lies in the plane $\{y=0\}$, it has a $45^\circ$ angle with the $\{ y=0\}$ plane and is in the $y \ge 0$ half-space. By the \textit{neighbourhood of an interesting  rectangle} we will mean a set that contains the rectangle and is relatively open in $\mathbb{R}\times [0,1]\times \mathbb{R}$.
\begin{claim}\label{c:lu}
    For every $\varepsilon > 0$ there exists a set $U$ such that for every possible direction $U$ contains a neighbourhood of an interesting rectangle in that direction, while it intersects the $\{y=h\}$ plane in a set of measure less than $\varepsilon$ for all $h$.
\end{claim}
\begin{proof}
    
Let $A$ be the set provided by Lemma \ref{t:folemma}. Fix an N large enough that 
\[ A \cap ((-N,N) \times [0,1] \times (-N,N))\]
contains an interesting rectangle in every direction.
 For any $\delta > 0$ let $A_\delta$ be $A + \{(x,0,z)| x^2 + z^2 <\delta\}$, where + denotes the Minkowski sum and define $A_0$ to be $A$. The set $A_\delta$ is open for all $\delta>0$, since $A$ is a union of planes (none of which are not parallel to $\{y=0\}$) and we replace each plane with an open set.
 Define
\[B_\delta = A_\delta \cap ([-N,N] \times [0,1] \times [-N,N]).\]
 The set $B_\delta$ is a relatively open in $[-N,N] \times [0,1] \times [-N,N]$ for all $\delta>0$, therefore $B_\delta$ contains a neighbourhood of an interesting rectangle in every possible direction.
 
 Let $f_\delta(h) = \lambda(B_\delta \cap \{y=h\})$. Using the fact that $B_0$ is compact, it is easy to verify that the function $f_\delta(h)$ is continuous in $\delta$. Since $f_0(h)=0, f_\delta(h)$ converges to 0 as $\delta$ tends to 0.
 
 We claim that
 \begin{equation}\label{e:suc}
      f_{\delta+t}(h_1) \ge f_{\delta}(h_2)
 \end{equation}
 whenever $|h_1 - h_2| < t$. In fact,
 \[(x, h_2, z) \in B_\delta \implies (x, h_1, z) \in B_{\delta+t}.\]
 Indeed, for every point $p$ in $B_\delta \cap \{ y= h_2 \}$ the set $A$ contains a plane that intersects $\{ y= h_2 \}$ in a line, which is at most distance $\delta$ away from $p$. This plane has a $45^\circ$ angle with the $\{ y= h_2 \}$ plane and intersects $\{ y= h_1 \}$ as well.

By \eqref{e:suc} and the continuity of $f_\delta(h)$ in $\delta$, we obtain that $f_\delta(h)$ is upper semicontinuous in $h$. Hence, $f_\frac{1}{n}$ is a sequence of upper semicontinuous functions on $[0,1]$, which is pointwise monotonically decreasing, and pointwise converges to 0. It is easy to prove that such a sequence must uniformly converge to 0. Therefore, there exists a $\delta_0 > 0$ such that $U=B_{\delta_0}$ has all the required properties.
\end{proof}

\subsection{Rotating an interesting rectangle}

Let $R$ be an interesting rectangle. Call a continuous motion $M$ of $R$ \textit{interesting} if at every moment it keeps $R$ interesting.
 \begin{lemma} \label{l:intmozg}
     For all $\varepsilon>0$ there exists an interesting motion of $R$, during which $R$ does a full a rotation, but $R \cap \{ y=t\}$ sweeps an area less than $\varepsilon$ for all $t$.
 \end{lemma}

\begin{proof}
    
We will use the idea of Pál joins to move between translated copies.
    We claim that for any interesting rectangle $R'$ parallel to $R$ there exists an interesting motion of $R$ during which $R\cap\{y=t\}$ sweeps less than $\varepsilon$ area for all $t$ and $R$ is translated to $R'$.

    There is a direction in which we can translate $R$ in such a way, that it sweeps 0 area at every height, call this its free direction. We will move $R$ in an N-like shape (Fig.~\ref{fig:pj}). We translate it in its free direction, there we rotate it by a small angle, then we translate it in its free direction, and rotate it in the opposite direction by the same angle. If we translate it far enough, then the angle of rotation can be arbitrarily small, and so the area swept at each height will be small.

\begin{figure}
    \centering
    \includegraphics[width=1\linewidth]{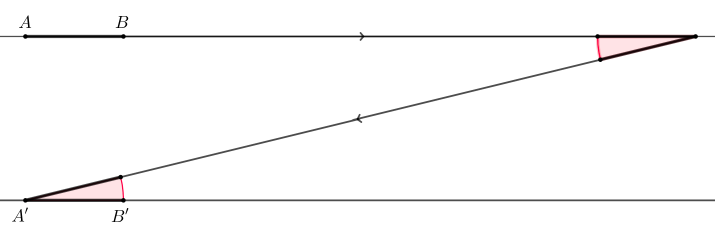}
    \caption{The motion of $R$ in the $y=0$ plane during a Pál join.}
        \label{fig:pj}
\end{figure}

Let $U$ be the set given by Claim~\ref{c:lu}. For any possible direction $d$, $U$ contains the neighbourhood of an interesting rectangle $R_d$, which is in the direction of $d$. For each $R_d$, there is an interesting motion, that rotates the rectangle a small amount within $U$. Since $S^1$ is compact, we can choose finitely many directions, whose neighbourhoods cover all directions. We have seen that we can move between parallel copies of $R$ while sweeping an arbitrarily small area at all heights, therefore these will only add an arbitrarily small area to the arbitrarily small area of $U$ at each height.
\end{proof} 
\subsection{Proof of the main result}
\begin{proof}[Proof of Theorem~\ref{t:mainresult}]
    
Let $M$ be the motion of an interesting rectangle $R$ given by Lemma~\ref{l:intmozg}. Let $T = R \cap \{ y=0 \}$. If at every moment we project $R$ onto the $\{ y=0\}$ plane, then we get $M'$, a motion of the unit square. Observe that during $M'$ the segment at $t$ distance from $T$ sweeps the same area, as $R \cap \{ y=t\}$ does during $M$.

\end{proof}

\section{Sets with the strong Kakeya property in $\mathbb{R}^3$}\label{s:sk}

Recall that we say that a compact set $K\subset \mathbb{R}^3$ is \textit{cylinderlike} if there exists $n\in \mathbb{N}$ such that for almost all $t$ the set $\{y=t\}\cap K$ can be covered by $n$ vertical lines.

\begin{lemma}\label{t:cylinderlike}
    If $\phi$ is a rotation around the $y$-axis, $K$ is cylinderlike, then $K$ can be moved to $\phi(K)$ in an arbitrarily small volume.
\end{lemma}

\begin{proof} We will look at $K$ from the direction of the $y$-axis and give a motion that keeps the $y$-coordinate of every point constant. The set $K$ is bounded, so its projection onto the $\{y=0\}$ plane can be covered by a square. Apply Theorem \ref{t:mainresult} to this square. This naturally induces a motion on $K$. For almost all planes perpendicular to the $y$-axis the swept area will be less than $n\varepsilon$, since if a plane contains $n$ vertical segments, then the motion of each segment corresponds to the motion of a vertical segment of the square and these all sweep an area less than $\varepsilon$. Applying Fubini's theorem we are done.

\end{proof}

\begin{remark}
    It is easy to check that the proof of the above lemma also works for \emph{general cylinderlike} sets defined as follows: For a compact set $K\subset \mathbb{R}^3$ let $n(t)$ be the minimum (possibly infinite) number of vertical lines required to cover $K \cap \{ y= t\}$. Say that $K$ is \textit{general cylinderlike} if $\int_\mathbb{R} n(t) dt < \infty$.
\end{remark}

Recall that we say that a compact set $K$ is \emph{cylinderlike from direction $d$} if there is an orthogonal coordinate system, in which the $y$-axis is parallel to $d$, and $K$ is cylinderlike.

By Lemma $\ref{t:cylinderlike}$: If $K$ is cylinderlike from direction $d$, then $K$ can be rotated around $d$ in an arbitrarily small volume.

 \begin{proof}[Proof of Theorem~\ref{t:fogantyu}] Translations can be done again with Pál joins: If $K$ is cylinderlike from some direction, then there exists a direction in which K can be translated in 0 volume. We can bring it very far from its original position sweeping 0 volume, there we rotate it by a very small angle, translate it in the free direction and rotate it back by the same angle. Therefore, we can move $K$ to an arbitrary translated position in an arbitrarily small area.
 
 Let $\phi$ be an orientation preserving isometry of $\R^3$ such that $\phi(K)$ is the desired position of $K$. By using translations we can suppose that the origin is a fixed point of $\phi$. Take unit vectors $v_1$ and $v_2$ parallel to $d_1$ and $d_2$ respectively. This can be done in a way such that their distance on $S^2$ is $t\le \frac\pi2$. Since $v_1$ and $v_2$ are non-parallel, the vectors $\phi(v_1), \phi(v_2)$ uniquely determine $\phi$. From this point forward every distance is the arc distance on $S^2$. By Lemma~\ref{t:cylinderlike} we can rotate $K$ around the vectors $v_1, v_2$ in an arbitrarily small volume. Therefore, it is enough to solve the following problem: Take $S^2$ and stick two needles ($n_1$ and $ n_2$) into it, at endpoints of $v_1$ and $v_2$ respectively. The following step is allowed: take a needle and rotate the sphere around it. The other needle moves accordingly. We need to prove that we can get the needles to any pair of fixed points $p_1$ and $p_2$, which are $t$ distance apart, in a finite number of steps.

\emph{First case: Suppose $n_1$ and $p_1$ are further than $t$ apart.} Then we rotate around $n_1$ in a way such that $n_2$ moves onto a short arc between $n_1 $ and $ p_1$. Then we rotate around $n_2$ by $\pi$. If the original distance of $n_1$ and $p_1$ was $s$, then their new distance after this step will be $|s-2t|$. It is clear that after a finite number of steps $n_1$ and $ p_1$ will be at most $t$ distance apart.

 \emph{Second case: Suppose $n_1, p_1$ are at most $t$ apart.} We rotate around $n_1$ in a way such that $n_2$ moves onto the circle centred around $p_1$ of radius $t$. We can rotate around $n_2$ in such a way that $n_1$ moves into $p_1$. Now we can rotate $n_2$ into $p_2$ by rotating around $n_1$.
 \end{proof}

\begin{cor}\label{c:axs}
    If for a compact set $A \subset \mathbb{R}^2$ there exists an $n\in\mathbb{N}$ and non-parallel directions $d_1$ and $d_2$, such that every line perpendicular to $d_i$ intersects $A$ in at most $n$ points for $i\in\{1,2\}$, then $A\times[0,1]$ has the strong Kakeya property.
\end{cor}
\begin{proof}
    The set $A\times[0,1]$ is cylinderlike from directions $d_1$ and $d_2$, therefore by Theorem~\ref{t:fogantyu} it has the strong Kakeya property.
\end{proof}

\begin{cor}
    If, in a fixed coordinate system, a compact $A$ can be covered by a finite union of graphs of Lipschitz functions, then $A\times[0,1]$ possesses the strong Kakeya property.
\end{cor}
\begin{proof}
    The graph of a Lipschitz function intersects every steep enough line in one point, therefore we can use Corollary~\ref{c:axs}.
\end{proof}

\begin{cor}\label{r:ert}
    If, in a fixed coordinate system, a compact $A$ is such that it can be covered by the graphs of a finite number of monotonic functions, then $K$ possesses the strong Kakeya property.
\end{cor}
\begin{proof}
    Under such conditions $A$ intersects every horizontal and vertical line in at most $n$ points, where $n$ is the number of monotonic functions required to cover $A$. Again Corollary~\ref{c:axs} can be applied.
\end{proof}

\begin{cor}
    The lateral surface of a cylinder has the strong Kakeya property. Moreover, the finite union of parallel lateral cylinder surfaces possesses the strong Kakeya property.
\end{cor} 

\begin{cor}
    If a compact $K\subset\R^3$ can be covered by a finite set of planes that have normal vectors in a common plane, then $K$ has the strong Kakeya property.
\end{cor}

\end{document}